\documentclass[a4,11pt]{article}

\usepackage{amssymb,amsfonts,amsmath,latexsym,color, amscd} 
\usepackage{url}
\usepackage{mathrsfs}

\usepackage[dvips]{psfrag,graphicx} 

\usepackage{cases}

\newtheorem{theorem}{Theorem}[section]
\newtheorem{lemma}[theorem]{Lemma}

\newtheorem{proposition}[theorem]{Proposition}

\newtheorem{definition}[theorem]{Definition}
\newtheorem{example}{Example}[section]
\newenvironment{proof}{{\par\addvspace{0.1cm}\noindent \bf Proof. }}{\hfill$\Box$\par\medskip}

\newtheorem{remark}[theorem]{Remark}

\numberwithin{equation}{section}

\def\a{\alpha}
\def\e{\varepsilon}
\def\R{\Re\mathfrak{e} \,}
\def\vect#1{\mbox{\boldmath $#1$}} 
\def\RR{\mathbb{R}}
\def\CC{\mathbb{C}}

\def\NN{\mathbb{N}}

\def\Res{\mbox{\rm Res}}

\def\const{\frac{\mu_0}{4\pi}}
\def\cii{\frac{\mu_0}{2\pi}}
\def\ciil{\frac{\mu_0\mathcal{L}}{2\pi}}
\def\ciilz{\frac{\mu_0\mathcal{L}}{2\pi(z+1)}}

\begin{document}

\title{Regularization of self inductance}

\author{Jun O'Hara\footnote{Supported by JSPS KAKENHI Grant Number 16K05136.}}
%
%
\maketitle

\begin{abstract}
We introduce several methods to define the self-inductance of a single loop as the regularization of divergent integrals which we obtain by applying Neumann (or Weber) formula for the mutual inductance of a pair of loops to the case when two loops are identical. 
%
\end{abstract}




\section{Introduction}
%
%
Inductance is an important notion in classical electrodynamics. 
Suppose two closed circuits $\Gamma_1$ and $\Gamma_2$, mutually disjoint, carry currents $I_1$ and $I_2$ respectively. 
The magnetic field energy $U_{12}$ of the system is given by $U_{12}=I_1I_2L_{12}$, where the coefficient $L_{12}$ is called the {\em mutual inductance}. 
According to Neumann formula (\cite{N}), the mutual inductance can be expressed by 
\begin{equation}\label{f_Neumann_link}
L_N(\Gamma_1,\Gamma_2):=\const \int_{\Gamma_1}\int_{\Gamma_2}\,\frac{\,d\vect x_1\cdot d\vect x_2\,}{|\vect x_1-\vect x_2|},
\end{equation}
where $\mu_0=4\pi\times 10^{-7} \mbox{H/m}$ is the permeability of free space. 
There is another expression of the mutual inductance, Weber's formula; 
by 
\begin{equation}\label{f_Weber_link}
L_W(\Gamma_1,\Gamma_2)
:=\const \int_{\Gamma_1}\int_{\Gamma_2}\,\frac{\,(\vect{\hat r}_{12}\cdot d\vect x_1)(\vect{\hat r}_{12}\cdot d\vect x_2)\,}{|\vect x_1-\vect x_2|},
\end{equation}
where $\vect r_{12}:=\vect x_2-\vect x_1$ and ${\vect{\hat r}}_{12}:=\vect r_{12}/|\vect r_{12}|$. 
It is known that these two expressions are equivalent. The reader is referred to \cite{Da} p.230 for the proof. 
The equivalence also follows from \eqref{f_lemma}.


If we put $\Gamma_1=\Gamma_2$ in \eqref{f_Neumann_link} (or \eqref{f_Weber_link}) the integral diverges in logarithmic order because of the contribution of a neighbourhood of the diagonal set. 
There have been some studies to avoid this divergence difficulty. 
Bueno and Assis (\cite{BA}) used higher dimensional objects, i.e., a surface or a $3$-dimensional wire thicked around a core loop $\Gamma$ instead of $\Gamma$ itself. They also showed that self-inductances obtained from Neumann formula and Weber formula are equivalent. 
Dengler (\cite{D}) replace a neighbourhood of a point on the loop by a short straight wire segment to avoid the divergence of self-energy. 

In this article, we propose a new definition of the self-inductance of a single loop by regularizing $L_N(\Gamma, \Gamma)$ (and $L_W(\Gamma, \Gamma)$) without approximating the loop by higher dimensional objects (Theorem \ref{theorem1}). 
To avoid the divergence of self-energy, we use two methods studied in the theory of generalized functions, {\em Hadamard regularization} and regularization via {\em analytic continuation} (see, for example, \cite{GS} or \cite{Z}). 
The same methods have been used to define the {\sl M\"obius energy} of a knot by regularizing 
$\int_\Gamma\int_\Gamma{{|\vect x-\vect y|}^{-2}}{\,dx\,dy}$ 
in \cite{O1} and \cite{B} respectively, and in general, the {\sl Riesz energy} of a manifold $M$ by regularizing $\int_M\int_M{{|\vect x-\vect y|}^{\alpha}}{\,dx\,dy}$ in \cite{OS2}. 

There is no difference between making a choice of two methods, but unlike in the result by Bueno and Assis, the difference between making a choice of Neumann and Weber formulae is equal to $(\mu_0/2\pi)$ times the length of the loop. 

We then show that our regularized self-inductance coincides with the regularization of the mutual inductance of a pair of close parallel loops as the distance between them tends to $0$ (up to addition by a constant depending on the length of the loop) (Theorem \ref{theorem_parallel}). 
The same method has been used in \cite{OS1} (see Subsection \ref{subsection_OS1}). 

Finally we study regularized self-inductance of solenoids (i.e. coils that wind cylinders) (Theorem \ref{thm_solenoids}). 
If we expand it in a series in the number $n$ of turns per unit length, the dominant term is of the order of $n^2$, and the coeffcient does not depend on the choice of Neumann or Weber formula to start with. 
This coefficient, when the length $\ell$ of the cylinder goes to infinity, is asymptotic to $\mu_0 \pi r^2$ times $\ell$ plus constant. 
Thus, by taking the asymptotics as $n$ and $\ell$ go to infinity, we can deduce that the dominat term of self-inductance of a solenoid is equal to $\mu_0 \pi r^2 n^2 \ell$ which fits a well-known formula, which implies the validity of our definition of regularized self-inductance.

\section{Regularization of self-inductance} 
%
\subsection{Regularization with a single loop}
We start with explaining the idea of Hadamard regularization of a divergent integral in a general setting. 
Take an ``$\e$-neibourhood'' of the set where the integrand blows up, restrict the integration to the complement of it, expand the result in a Laurent series in $\e$ (possibly with a $\log$ term or terms with non-integer powers), and take the constant term. The constant term is called {\em Hadamard's finite part} of the integral.  
The terms with negative powers (and a log term if exists) are called the {\em counter terms}. 

\smallskip
Let $\Gamma$ be a smooth\footnote{We assume the smoothness for the sake of simplicity. In fact, $\Gamma$ being $C^1$ is enough.} simple loop in $\RR^3$ with length $\mathcal{L}$ parametrized by $\gamma(s)$ by the arc-length. 
\begin{theorem}\label{theorem1} 
The self-inductance can be regularized in the following way. 
\begin{enumerate}
\item Hadamard regularization of $L_N(\Gamma,\Gamma)$ and $L_W(\Gamma,\Gamma)$ can be carried out as follows. Let $\Delta_\e$ be an ``$\e$-neighbourhood'' of the diagonal set with respect to the distance in $\RR^3$: 
\[\Delta_\e:=\{(\vect x_1, \vect x_2)\in\RR^3\times\RR^3\,:\,|\vect x_1-\vect x_2|\le\e\}.\]
There exist the limits 
\begin{eqnarray}
 \displaystyle  H_N(\Gamma)&:=&\displaystyle \lim_{\e\to0^+}\left(\const \iint_{(\Gamma\times \Gamma)\setminus\Delta_\e}\,\frac{\,d\vect x_1\cdot d\vect x_2\,}{|\vect x_1-\vect x_2|}+\ciil\log\e\right), \label{Hadamard_N} \\[1mm]
\displaystyle H_W(\Gamma)&:=&\displaystyle \lim_{\e\to0^+}\left(\const\iint_{(\Gamma\times \Gamma)\setminus\Delta_\e}\,\frac{\,(\vect{\hat r}_{12}\cdot d\vect x_1)(\vect{\hat r}_{12}\cdot d\vect x_2)\,}{|\vect x_1-\vect x_2|}+\ciil\log\e\right). \label{Hadamard_W}
\end{eqnarray}
%
%
\item Fix the loop $\Gamma$ and consider  
\begin{eqnarray}
\displaystyle F_N(z)&:=&\displaystyle \const \iint_{\Gamma\times \Gamma}{|\vect x_1-\vect x_2|}^z \,d\vect x_1\cdot d\vect x_2, \label{F_N(z)} \\[1mm]
\displaystyle F_W(z)&:=&\displaystyle \const \iint_{\Gamma\times \Gamma}{|\vect x_1-\vect x_2|}^z \,(\vect{\hat r}_{12}\cdot d\vect x_1)(\vect{\hat r}_{12}\cdot d\vect x_2) \label{F_W(z)}
\end{eqnarray}
as functions of a complex variable $z$. 
Then both $F_N(z)$ and $F_W(z)$ are well-defined and holomorphic on $\{z\in\CC\,:\,\R z>-1\}$. 
The domains of $F_N(z)$ and $F_W(z)$ can be extended to the whole complex plane $\CC$ by analytic continuation to make meromorphic functions with possible simple poles at negative odd integers. Let them be denoted by the same symbols $F_N(z)$ and $F_W(z)$. The first residues are given by 
\[
\Res(F_N,-1)=\Res(F_W,-1)=\ciil.
\]
There exist the limits 
\begin{eqnarray}
\displaystyle A_N(\Gamma)&:=&\displaystyle 
%
\lim_{z\to-1}\left(F_N(z)-\ciilz\right), \label{analytic_continuation_N} \nonumber \\[1mm]
\displaystyle A_W(\Gamma)&:=&\displaystyle 
%
\lim_{z\to-1}\left(F_W(z)-\ciilz\right). \label{analytic_continuation_W} \nonumber
\end{eqnarray}
\item In each case starting with Neumann formula or Weber formula, there is no difference between making a choice of two methods of regularization, namely, $H_N(\Gamma)=A_N(\Gamma)$ and $H_W(\Gamma)=A_W(\Gamma)$. 
\item The difference between making a choice of Neumann formula and Weber formula is equal to $(\mu_0/2\pi)$ times the length of the loop, namely, 
\begin{equation}\label{f_N=W_knot}
H_W(\Gamma)=H_N(\Gamma)+\cii \, \mathcal{L}. \nonumber
\end{equation}
\end{enumerate}
\end{theorem}

\begin{definition} \rm 
Let us call $H_N(\Gamma)=A_N(\Gamma)$ and $H_W(\Gamma)=A_W(\Gamma)$ the {\em regularized self-inductances} of $\Gamma$ {\em in the sense of Neumann and Weber} respectively. 
\end{definition}

The proof of Theorem \ref{theorem1} is divided in several steps. 

Assertions (1),(2) and (3) follow from the argument in Section 3 of \cite{OS2} with slight modification. 

\medskip
Let $\vect t_{\vect x_j}$ be the unit tangent vector to $\Gamma$ at $\vect x_j$ and $\theta_j$ $(j=1,2)$ be the angle between $\vect t_{\vect x_j}$ 
and $\vect r_{12}$. 
We remark that the numerator of integrand of \eqref{Hadamard_W} can be experssed as 
\[
(\vect{\hat r}_{12}\cdot d\vect x_1)(\vect{\hat r}_{12}\cdot d\vect x_2) 
=\cos\theta_1\cos\theta_2\,dx_1dx_2\,.
\]

Let us give the relation between the arc-length parameter and the {\sl chord length} (the distance in $\RR^3$) of a nearby point from a fixed point on the knot explicitly. 

\begin{lemma}\label{lemma_series_s<->t}
Let $\vect x_1$ be a point on $\Gamma$. Suppose a nearby point $\vect x_2$ on $\Gamma$ is expressed by the arc-length parameter $s$ from $\vect x_1$. Namely, $\vect x_1=\gamma(s_1)$ for some $s_1$ and $\vect x_2=\gamma(s_1+s)$ $(-\mathcal{L}/2<s<\mathcal{L}/2)$. 
Let $t$ be the chord length (the distance in $\RR^3$) between $\vect x_1$ and $\vect x_2$. 
Then\footnote{The last terms of \eqref{estimate_t_by_s}, \eqref{estimate_s_by_t} and \eqref{estimate_numerator_N} are not necessary in this article. 
We put them here to illustrate $\psi_{\vect x_1}^{(4)}(0)=0$ in \eqref{f_series_psi_in_t} in Lemma \ref{lemma_series_varphi}. }
\begin{eqnarray}
t&=&\displaystyle |s|\left(1-\frac{\kappa^2}{24}\,s^2-\frac{\kappa\kappa'}{24}s^3+O(s^4)\right), \label{estimate_t_by_s}  \\
s&=&\displaystyle \left\{
\begin{array}{ll}
\displaystyle \phantom{-}t+\frac{\kappa^2}{24}\,t^3+\frac{\kappa\kappa'}{24}t^4+O(t^5) &\hspace{0.5cm} (s\ge0), \\ [2mm]
\displaystyle -t-\frac{\kappa^2}{24}\,t^3+\frac{\kappa\kappa'}{24}t^4+O(t^5) &\hspace{0.5cm} (s<0), 
\end{array}
\right. \label{estimate_s_by_t}
\end{eqnarray}
where $\kappa$ and $\kappa'$ are the curvature and its derivative with respect to $s$ at point $\vect x_1$, i.e., at $s=s_1$. 
Furthermore the numerators of integrands of \eqref{Hadamard_N} and \eqref{Hadamard_W} can be estimated by 
\begin{equation}\label{estimate_numerator_N}
\displaystyle \vect t_{\vect x_1}\cdot\vect t_{\vect x_2}
=1-\frac{\kappa^2}2s^2-\frac{\kappa\kappa'}2s^3+O(s^4)
=\left\{
\begin{array}{ll}
\displaystyle 1-\frac{\kappa^2}2t^2-\frac{\kappa\kappa'}2t^3+O(t^4) & \hspace{0.4cm}(s\ge0)\\[2mm]
\displaystyle 1-\frac{\kappa^2}2t^2+\frac{\kappa\kappa'}2t^3+O(t^4) & \hspace{0.4cm}(s<0),
\end{array}
\right.
\end{equation}

\begin{equation}\label{estimate_numerator_W}
(\vect{\hat r}_{12}\cdot \vect t_{\vect x_1})(\vect{\hat r}_{12}\cdot \vect t_{\vect x_2}) =
\displaystyle 1-\frac{\kappa^2}4s^2+O(s^3)
=1-\frac{\kappa^2}4t^2+O(t^3). 
\end{equation}
\end{lemma}

\begin{proof} 
Computation using Frenet-Serret's formula implies that $\gamma(s_1+s)$ can be expressed with respect to the Frenet-Serret frame at $\gamma(s_1)$ as 
\[
\begin{array}{rcrrrr}
\xi&=&s&&\displaystyle -\frac{\kappa^2}6s^3 &\displaystyle -\frac{\kappa\kappa'}8s^4+O(s^5)
\\[4mm]%
\eta&=&&\displaystyle \frac\kappa2s^2 &\displaystyle  +\frac{\kappa'}6s^3 &\displaystyle -\frac{\kappa^2+\kappa\tau^2-\kappa''}{24}s^4+O(s^5)
\\[4mm]%
\zeta&=&&&\displaystyle \frac{\kappa\tau}6s^3 &\displaystyle +\frac{2\kappa'\tau+\kappa\tau'}{24}s^4+O(s^5),
\end{array}
\]
where $\tau$ means the torsion. 
Then, \eqref{estimate_t_by_s} can be obtained by substituting the above series expansion to 
\[
t=|\gamma(s_1+s)-\gamma(s)|={\left[\left(\gamma(s_1+s)-\gamma(s)\right)\cdot\left(\gamma(s_1+s)-\gamma(s)\right)\right]}^{\frac12}\,.
\]
With this frame $\vect{\hat r}_{12}$ is expressed as
\[
\vect{\hat r}_{12}
=\left(
1-\frac{\kappa^2}8s^2,\,
\frac\kappa2s+\frac{\kappa'}6s^2,\,\frac{\kappa\tau}6s^2
\right)+O(s^3).
\]
\end{proof}

Next lemma is needed for the proof of (2). 

Let $B_t(\vect x)$ denote the $3$-ball with center $\vect x$ and radius $t$. 
Let $d$ be the diameter of $\Gamma$. 

\begin{lemma}\label{lemma_series_varphi}
Fix $\vect x_1$ on $\Gamma$ and put 
\[
\psi_{\vect x_1}(t):=\int_{\Gamma\cap B_t(\vect x_1)}\vect t_{\vect x_1}\cdot \vect t_{\vect x_2} \,dx_2
\hspace{0.8cm}(0\le t\le d). 
\]
Then $\psi_{\vect x_1}(t)$ extends to a smooth function on $(-d,d)$, denoted by the same symbol, with $\psi_{\vect x_1}(-t)=-\psi_{\vect x_1}(t)$. 
Put 
\[
\varphi(t):=\psi_{\vect x_1}'(t). 
\]
Then $\varphi^{(2j-1)}(0)=0$ for $j\in\NN$. 
To be precise, $\varphi(t)$ can be expanded in a series in $t$ as
\begin{equation}\label{series_varphi}
\varphi(t)=2-\frac{3\kappa^2}4t^2+O(t^4).
\end{equation}
%
\end{lemma}

\begin{proof}
The first half follows from Proposition 3.1 of \cite{OS2} by putting $\rho(x_1,x_2)=\vect t_{\vect x_1}\cdot\vect t_{\vect x_2}$. 
We remark that the assertion $\psi_{\vect x_1}^{(2j)}(0)=0$ $(j\in\NN)$ can be illustrated by the fact that $\int_{-\delta}^\delta u^{2j-1}\,du=0$. 

We prove \eqref{series_varphi} in what follows. 
The equalities \eqref{estimate_s_by_t} and \eqref{estimate_numerator_N} imply
\begin{equation}\label{f_series_psi_in_t}
\begin{array}{rcl}
\psi_{\vect x_1}(t)&=&\displaystyle 
\int_0^t\left(1-\frac{\kappa^2}2u^2-\frac{\kappa\kappa'}2u^3+O(u^4)\right)
\left(1+\frac{\kappa^2}8u^2+\frac{\kappa\kappa'}6u^3+O(u^4)\right)du
\\[4mm]
&&\displaystyle +\int_t^0\left(1-\frac{\kappa^2}2u^2+\frac{\kappa\kappa'}2u^3+O(u^4)\right)
\left(-1-\frac{\kappa^2}8u^2+\frac{\kappa\kappa'}6u^3+O(u^4)\right)du
\\[4mm]
&=&\displaystyle 2t-\frac{\kappa^2}4t^3+O(t^5), 
\end{array}
\end{equation}
%
and hence 
\[
\varphi(t)=2-\frac{3\kappa^2}4t^2+O(t^4),
\]
which completes the proof. 
\end{proof}

Finally, we give a lemma which is needed for the proof of (4) of Theorem \ref{theorem1}. 

\begin{lemma}\label{lemma}
Let $\omega_1$ be the $1$-form on $\Gamma$ given by $\omega_1=d\vect x_1\cdot{\vect{\hat r}}_{12}$. Then we have 
\begin{equation}\label{f_lemma}
\frac{\,d\vect x_1\cdot d\vect x_2\,}{|\vect x_1-\vect x_2|^\a}
=\a\,\frac{(\vect{\hat r}_{12}\cdot d\vect x_1)(\vect{\hat r}_{12}\cdot d\vect x_2)}{|\vect x_1-\vect x_2|^\a}-d\left(\frac{\omega_1}{|\vect x_1-\vect x_2|^{\a-1}}\right).
%
\end{equation}
\end{lemma}

\begin{proof} 
It can be obtained by modifying the proof of Proposition 4.11 of \cite{OS1} (cf. Proposition 5 of \cite{banchoff.pohl}). 

Suppose $(\vect x_1, \vect x_2)$ belongs to $\Gamma\times\Gamma\setminus\Delta$, where $\Delta$ is the diagonal $\Delta=\{(x,x)\,:\,x\in\RR^3\}$ (the case when $(\vect x_1, \vect x_2)\in\Gamma_1\times\Gamma_2$ can be proved in the same way). 
Let $\tau$ be the angle between the two oriented planes containing the line $\overline{\vect x_1\vect x_2}$ tangent to $\Gamma$ at $\vect x_1$ and $\vect x_2$ respectively. 

Let $\{\vect e_1, \vect e_2, \vect e_3\}$ be an orthonormal moving frame (along $\Gamma$) with $\vect e_1=\vect{\hat r}_{12}$, and $\vect e_3\bot T_{\vect x_1}\Gamma$. 
Let $\omega_i=d\vect x_1\cdot \vect e_i$ and $\omega_{ij}=d\vect e_i\cdot \vect e_j$. 
Then there hold 
\[
\omega_2=\sin\theta_1\,d x_1, \>\>\omega_3=0, \>\> \omega_{12}=\cos\tau\sin\theta_2\,\frac{d x_2}{|\vect x_2-\vect x_1|}, 
\]
which implies
\begin{eqnarray}
d\omega_1&=&\omega_{12}\wedge\omega_{2}+\omega_{12}\wedge\omega_{3} 
=\displaystyle -\cos\tau\sin\theta_1\sin\theta_2\,\frac{dx_1\wedge dx_2}{|\vect x_2-\vect x_1|}. \label{f_lemma_forms}
\end{eqnarray}

\begin{enumerate}
\item The case $\a\ne1$. We have 
\[
\begin{array}{rcl}
\displaystyle \cos\theta_1\cos\theta_2\,\frac{dx_1\wedge dx_2}{|\vect x_2-\vect x_1|^\a}
&=&\displaystyle \frac{(\vect{\hat r}_{12}\cdot d\vect x_1)(\vect{\hat r}_{12}\cdot d\vect x_2)}{|\vect x_1-\vect x_2|^\a}\\[4mm]
&=&\displaystyle -\frac{d (|\vect x_2-\vect x_1|)\wedge\omega_{1}}{|\vect x_2-\vect x_1|^\a}
=\frac1{\a-1} \, d \left(\frac{1}{|\vect x_2-\vect x_1|^{\a-1}}\right)\wedge\omega_{1} \\[4mm]
&=&\displaystyle \frac1{\a-1} \, d \left(\frac{\omega_{1}}{|\vect x_2-\vect x_1|^{\a-1}}\right)
-\frac1{\a-1} \cdot \frac{1}{|\vect x_2-\vect x_1|^{\a-1}}\,d \omega_{1}\\[4mm]
&=&\displaystyle \frac1{\a-1}\cdot\cos\tau\sin\theta_1\sin\theta_2\,\frac{dx_1\wedge dx_2}{|\vect x_2-\vect x_1|^\a}
+\frac1{\a-1}\,d \left(\frac{\omega_{1}}{|\vect x_2-\vect x_1|^{\a-1}}\right),
\end{array}
\]
and therefore 
\[
\cos\tau\sin\theta_1\sin\theta_2\,\frac{dx_1\wedge dx_2}{|\vect x_2-\vect x_1|^\a}
=(\a-1)\cos\theta_1\cos\theta_2\,\frac{dx_1\wedge dx_2}{|\vect x_2-\vect x_1|^\a}
-d \left(\frac{\omega_{1}}{|\vect x_2-\vect x_1|^{\a-1}}\right).
\]
Since 
\begin{equation}\label{f_lemma_proof}
\frac{d\vect x_1\cdot d\vect x_2}{|\vect x_2-\vect x_1|^{\a}}=(\cos\theta_1\cos\theta_2+\cos\tau\sin\theta_1\sin\theta_2)\,\frac{dx_1\wedge dx_2}{|\vect x_2-\vect x_1|^{\a}},
\end{equation}
\eqref{f_lemma} follows. 

\item The case $\a=1$ follows from \eqref{f_lemma_forms} and \eqref{f_lemma_proof}.
\end{enumerate}
\end{proof}

\bigskip
\noindent
{\bfseries Proof of Theorem {\rm \bf \ref{theorem1}.}}
We prove only in the case of Neumann formula, as the argument goes parallel for Weber formula. 
We drop off the coefficient $\mu_0/(4\pi)$ in the proof to make formulae shorter and simpler. 

\medskip
(1) Since 
\[
\iint_{(\Gamma\times \Gamma)\setminus\Delta_\e}\,\frac{\,d\vect x_1\cdot d\vect x_2\,}{|\vect x_1-\vect x_2|}
=\int_\Gamma\left(\int_{\Gamma\setminus B_\e(\vect x_1)} \frac{\vect t_{\vect x_1}\cdot \vect t_{\vect x_2}}{|\vect x_1-\vect x_2|} \,dx_2\right)dx_1
\]
the regularization process can be reduced to that of $\int_\Gamma {|\vect x_1-\vect x_2|}^{-1} \,(\vect t_{\vect x_1}\cdot \vect t_{\vect x_2}) \,dx_2$. 
The assertion follows from \eqref{estimate_t_by_s} and \eqref{estimate_numerator_N} since the integrand can be estimated by 
\[
\frac{\vect t_{\vect x_1}\cdot \vect t_{\vect x_2}}{|\vect x_1-\vect x_2|}
=\frac1{|s|}+O(s). 
\]

\medskip
(2) 
Since 
\[
\iint_{\Gamma\times \Gamma}{|\vect x_1-\vect x_2|}^z \,d\vect x_1\cdot d\vect x_2
=\int_\Gamma\left(\int_\Gamma {|\vect x_1-\vect x_2|}^z \,(\vect t_{\vect x_1}\cdot \vect t_{\vect x_2}) \,dx_2\right)dx_1, 
\]
the regularization of $F_N(z)$ can be reduced to that of $\int_\Gamma {|\vect x_1-\vect x_2|}^z \,(\vect t_{\vect x_1}\cdot \vect t_{\vect x_2}) \,dx_2$. 

By the coarea formula (a kind of geometric version of Fubini's theorem), we have 
\[\begin{array}{rcl}
\displaystyle \int_\Gamma {|\vect x_1-\vect x_2|}^z \,(\vect t_{\vect x_1}\cdot \vect t_{\vect x_2}) \,dx_2 
&=&\displaystyle \int_0^d t^z \psi_{\vect x_1}'(t)\,dt
=\int_0^d t^z \varphi(t)\,dt. 
\end{array}
\]
Put 
\[
f(z):=\int_0^d t^z\,\varphi(t)\,dt, 
\]
which converges for $\Re z>-1$. Let $k$ be any natural number, and consider the right hand side\footnote{Here we divide the domain of integration at $1$. When $d<1$ we can take any $d_0$ with $0<d_0<d$ instead of $1$. The argument goes parallel.} of 
\begin{equation}\label{GS}
\begin{array}{rcl}
\displaystyle \int_0^d t^z\,\varphi(t)\,dt 
&=& \displaystyle \int_1^d t^z\,\varphi(t)\,dt
+\int_0^1t^z\left[\varphi(t)-\varphi(0)-\varphi'(0)t-\dots -\frac{\varphi^{(k-1)}(0)}{(k-1)!}\,t^{k-1}\right]\,dt 
 \\[4mm]
&& \displaystyle +\sum_{j=1}^{k}\int_0^1\frac{\varphi^{(j-1)}(0)}{(j-1)!}\,t^{z+j-1}\,dt.
\end{array}
\end{equation}
The first term is a holomorphic function of $z$. 
The integrand of second term can be estimated by $t^{z+k}$, hence the integral converges for $\Re z>-k-1$. 
Since 
\[
\int_0^1\frac{\varphi^{(j-1)}(0)}{(j-1)!}\,t^{z+j-1}\,dt=\frac{\varphi^{(j-1)}(0)}{(j-1)!\,(z+j)} \hspace{0.6cm}(z\ne-j)
\]
$f(z)$ is a meromorphic function on $\Re z>-k-1$ possibly with simple poles at $z=-1,\dots,-k$ with residues given by 
\begin{equation}\label{residue_t^wphi}
\Res(f,-j)=\frac{\varphi^{(j-1)}(0)}{(j-1)!} \hspace{0.6cm}(j=1,\dots,k)
\end{equation}
(\cite{GS} Ch.1, 3.2). 

Since $k\in\NN$ is arbitrary as $\varphi$ is smooth, and $\varphi^{(2i)}(0)=0$ $(i=0,1,2,\dots)$ by Lemma \ref{lemma_series_varphi}, this proves that $F_N(z)$ is a meromorphic function with possible simple poles at negative odd integers. 
Since $\varphi(0)=2$ by \eqref{series_varphi}, the residue at $z=-1$ is given by 
\[
\Res (F_N,-1)=\const\int_\Gamma \varphi(0)\, dx_1=\ciil. 
\]

\medskip
(3) The above argument can be paraphrased into Hadamard regularization as follows. 
Putting $k=1$ and $z=-1$ in \eqref{GS}, one obtains 
\begin{eqnarray}
\displaystyle \int_{\Gamma\setminus B_\e(\vect x_1)} \frac{\vect t_{\vect x_1}\cdot \vect t_{\vect x_2}}{|\vect x_1-\vect x_2|} \,dx_2
&=& \displaystyle \int_\e^d t^{\,-1}\,\varphi(t)\,dt \nonumber\\
&=& \displaystyle \int_1^d t^{\,-1}\,\varphi(t)\,dt
+\int_\e^1t^{\,-1}\left[\varphi(t)-\varphi(0)\right]\,dt 
\displaystyle +{\varphi(0)}\int_\e^1t^{\,-1}\,dt, \hspace{0.8cm}{ } \label{GS_Hadamard_z=-1}
\end{eqnarray}
%
Since 
\[
\int_0^1t^z\,dt=\frac1{z+1} \hspace{0.8cm}(z\ne-1),
\]
the residue of $z\mapsto \int_0^1t^z\,dt$ at $z=-1$ is equal to $1$. 
Camparing \eqref{GS} with $k=1$ and \eqref{GS_Hadamard_z=-1}, the first two terms in the right hand sides coincide since the integrals converge. 
The equality $H_N(\Gamma)=A_N(\Gamma)$ follows from 
\[
\lim_{z\to-1}\left(\int_0^1t^z\,dt-\frac1{z+1}\right)
=\lim_{\e\to0^+}\left(\int_\e^1\frac{dt}t+\log\e\right).
\]

\medskip
(4) For a sufficiently small $\e>0$, $B_\e(\vect x)\cap\Gamma$ consists of a single curve segment for any $\vect x\in\Gamma$. We remark that the condition is satisfied if $\e$ is smaller than the {\em thickness} of $\Gamma$ (see, for exmaple, \cite{LSDR} for the definition of thickness of a knot). 
Let $\vect x_{\e,+}$ and $\vect x_{\e,-}$ be the endpoints of the curve $B_\e(\vect x)\cap\Gamma$, where we assume $\vect x_{\e,+}$ is a little bit ahead of $\vect x$ with respect to the orientation of $\Gamma$. 

Then the boundary of $\Gamma\times\Gamma\setminus\Delta_\e$ consists of two disjoint curves on $\Gamma\times \Gamma$ given by 
\[
\Gamma_{\e,\pm}:=\{(\vect x, \vect x_{\e,\pm})\,:\,\vect x\in\Gamma\}.
\]
Suppose $\Gamma_{\e,\pm}$ are endowed with the same orientation as $\Gamma$. 
Then the boundary of $\Gamma\times\Gamma\setminus\Delta_\e$ is given by 
\[
\partial\left(\Gamma\times\Gamma\setminus\Delta_\e\right)=\Gamma_{\e,+}\cup(-\Gamma_{\e,-}),
\]
where $-\Gamma_{\e,-}$ is endowed with the reverse orientation. 
Since on $\Gamma_{\e,+}$ and $\Gamma_{\e,-}$
\[
\omega_1=d\vect x_1\cdot\frac{\vect x_2-\vect x_1}{|\vect x_2-\vect x_1|}=\left(1+O(\e^2)\right)\,dx_1,
\]
we have 
\[
\begin{array}{rcl}
\displaystyle \iint_{(\Gamma\times \Gamma)\setminus\Delta_\e} \left(\frac{\,(\vect{\hat r}_{12}\cdot d\vect x_1)(\vect{\hat r}_{12}\cdot d\vect x_2)\,}{|\vect x_1-\vect x_2|}
-\frac{\,d\vect x_1\cdot d\vect x_2\,}{|\vect x_1-\vect x_2|}\right)
&=&\displaystyle \iint_{(\Gamma\times \Gamma)\setminus\Delta_\e} d\omega_1\\[5mm]
&=&\displaystyle \int_{\Gamma_{\e,+}\cup\,(-\Gamma_{\e,-})}\omega_1 \\[5mm]
&=&\displaystyle 2\mathcal{L}+O(\e^2).
\end{array}
\]
\vspace{-0.5cm}
\begin{flushright}
{\small{$\square$}}
\end{flushright}
\bigskip

We close this subsection with giving a property and an example of the self-inductance. 

\begin{proposition}
The regularized self-inductance behaves under homothety as follows. 
\[
%
H_N(\lambda\Gamma)=\lambda H_N(\Gamma)+\ciil \cdot \lambda\,(\log \lambda) \hspace{0.8cm}(\lambda>0). 
\]
\end{proposition}
This is immediate from the definition \eqref{Hadamard_N}. 

\begin{example} \rm 
The regularized self-inductance of a unit circle $\Gamma_\circ$ is given by 
\[
\begin{array}{rcl}
H_N(\Gamma_\circ)&=&\displaystyle 2\pi\lim_{\e\to0^+}
\left(\const \int_\e^{2\pi-\e}\frac{\cos\theta}{2\sin\frac\theta2}\,d\theta+\cii \log\e
\right) 
=\frac{(-1+\log2)\,\mu_0}\pi.
\end{array}
\]
\end{example}

\subsection{Regularization with parallel loops}

\begin{theorem}\label{theorem_parallel} 
Let $\Gamma$ be a smooth\footnote{We assume the smoothness for the sake of simplicity. In this case we need $C^4$.} simple loop with non-vanishing curvature. 
Let $\Gamma_\delta$ be a $\delta$-parellel curve given by $\Gamma_\delta=\{\vect x+\delta \vect n(\vect x)\,:\,\vect x\in \Gamma\}$, where $\vect n$ is the unit principal normal vector to $\Gamma$. Then 
\begin{equation}\label{f_paralle}
%
\lim_{\delta\to0^+}\left(L_N(\Gamma, \Gamma_\delta)+\ciil\log\delta\right)
=H_N(\Gamma)+\frac{(\log2)\,\mu_0\mathcal{L}}{2\pi}.
\end{equation}
\end{theorem}

We remark that if $\delta$ is smaller than the thickness of $\Gamma$ then $\Gamma_\delta\cap\Gamma=\emptyset$.

\begin{proof}
We drop off the coefficient $\mu_0/(4\pi)$ in the proof to make formulae shorter and simpler. 
The proof of Proposition 4.18 of \cite{OS1} goes parallel with slight modification. 
We estimate, fixing $\vect x_1$ on $\Gamma$, 
\begin{equation}\label{f_parallel_local}
\int_{\Gamma_\delta\cap B_\e(\vect x_1)}\frac{\vect t_{\vect x_1}\cdot \vect t_{\vect x_2}}{|{\vect x_1}-{\vect x_2}|}\,dx_2
\end{equation}
for $0<\delta\ll\e\ll1$.

Suppose $\Gamma$ is parametrized by the arc-length as $\Gamma=\{\gamma(s)\}_{0\le s\le\mathcal{L}}$. Then $\Gamma_\delta$ can be expressed as  $\Gamma_\delta=\{\gamma_\delta(s)\}_{0\le s\le\mathcal{L}}$, where 
$\gamma_\delta(s)=\gamma(s)+\delta \kappa(s)^{-1}\gamma''(s).$ 
Let $\vect x_1=\gamma(s_1)$ and $\kappa$ be the curvature of $\Gamma$ at $\vect x_1$. 
We have
\[
\begin{array}{rcl}
\displaystyle \gamma'(s_1)\cdot \gamma_\delta'(s_1+s)&=&\displaystyle (1-\kappa\delta)+O(1)\delta s+O(s^2),  \\[2mm]
\displaystyle |\gamma_\delta(s_1+s)-\gamma(s_1)|^2&=&\displaystyle  \left(\delta^2+(1-\kappa\delta)s^2\right)\left(1+O(1)s^2\right).
\end{array}
\]
Let $s_1\pm s_\pm$ be the values of parameter when $\Gamma_\delta$ passes through the sphere $\partial B_\e(\vect x_1)$. 
Then 
\[
s_\pm=\pm\sqrt{\frac{\e^2-\delta^2}{1-\kappa\delta}}+O(\e^3). 
\]
The above equalities imply that \eqref{f_parallel_local} can be estimated by 
\[\begin{array}{l}
\displaystyle \int_{-\sqrt{\frac{\e^2-\delta^2}{1-\kappa\delta}}\,+O(\e^3)}^{\sqrt{\frac{\e^2-\delta^2}{1-\kappa\delta}}\,+O(\e^3)}
\frac{\left[(1-\kappa\delta)+O(1)\delta s+O(1)s^2\right]\left(1+O(1)s^2\right)}
{\sqrt{\delta^2+(1-\kappa\delta)s^2}}\,ds \\[6mm]
=\displaystyle 2\int_0^{\sqrt{\frac{\e^2-\delta^2}{1-\kappa\delta}}}\>
\frac{1-\kappa\delta}{\sqrt{\delta^2+(1-\kappa\delta)s^2}}\,ds+O(\e^2)+o(\delta) \\[6mm]
=\displaystyle 2\sqrt{1-\kappa\delta}\,\log\left(
\frac{\e+\sqrt{\e^2-\delta^2}}{\delta}
\right)
+O(\e^2)+o(\delta) \\[5mm] 
=\displaystyle 2\log2+2\log\e-2\log\delta+O(\e). 
\end{array}\]
Snce $\delta\ll\e$, it follows that 
\[
\begin{array}{rcl}
\displaystyle \int_\Gamma\int_{\Gamma_\delta}\frac{d\vect x_1\cdot d\vect x_2}{|\vect x_1-\vect x_2|}
&=&\displaystyle \int_\Gamma\left(
\int_{\Gamma_\delta\cap B_\e(\vect x_1)}\frac{\vect t_{\vect x_1}\cdot \vect t_{\vect x_2}}{|{\vect x_1}-{\vect x_2}|}\,dx_2
+\int_{\Gamma_\delta\setminus B_\e(\vect x_1)}\frac{\vect t_{\vect x_1}\cdot \vect t_{\vect x_2}}{|{\vect x_1}-{\vect x_2}|}\,dx_2

\right)dx_1 \\[6mm]
&=&\displaystyle \int_\Gamma\left(2\log2+2\log\e-2\log\delta
+\int_{\Gamma\setminus B_\e(\vect x_1)}\frac{\vect t_{\vect x_1}\cdot \vect t_{\vect x_2}}{|{\vect x_1}-{\vect x_2}|}\,dx_2
\right)dx_1 +O(\e) \\[6mm]
&=&\displaystyle \iint_{(\Gamma\times\Gamma)\setminus\Delta_\e}\frac{d\vect x_1\cdot d\vect x_2}{|\vect x_1-\vect x_2|}
+2\mathcal{L}\log\e+2(\log2)\mathcal{L}-2\mathcal{L}\log\delta +O(\e), 
\end{array}
\]
which means 
\[
\int_\Gamma\int_{\Gamma_\delta}\frac{d\vect x_1\cdot d\vect x_2}{|\vect x_1-\vect x_2|}
+2\mathcal{L}\log\delta
-2(\log2)\mathcal{L}
=
\iint_{(\Gamma\times\Gamma)\setminus\Delta_\e}\frac{d\vect x_1\cdot d\vect x_2}{|\vect x_1-\vect x_2|}
+2\mathcal{L}\log\e
+O(\e).
\]
Taking the limit as $\e$ goes down to $0$ and multiplying the both sides by $\mu_0/(4\pi)$, we obtain \eqref{f_paralle}. 
%
\end{proof}

\subsection{Hadamard regularization in terms of the arc-length parameter}
We introduce formulae of regularized self-inductance obtained by Hadamard regularization in terms of the arc-length parameter. They are more suitable for numerical experiments, and one of them will be used in the next subsection. 

In \eqref{Hadamard_N} and \eqref{Hadamard_W} we used the {\em chord-length} i.e. the distance in $\RR^3$ to define an ``{\sl $\e$-neighbourhood}\,'' $\Delta_\e$ of the diagonal set $\Delta$ of $\Gamma\times\Gamma$, which we use in Hadamard regularization. 
However, since it is easier to express a curve by the arc-length than by the chord-length, it is useful to have the formulae of Hadamard regularization in terms of the arc-length parametrization. 
In this setting, $\Delta_\e$ should be replaced by 
\[
\widetilde\Delta_\e:=\{(\vect x_1, \vect x_2)\in\Gamma\times\Gamma\,:\,d_\Gamma(\vect x_1,\vect x_2)\le\e\},
\]
where $d_\Gamma$ is the arc-length along $\Gamma$. Nevertheless, the counter terms and Hadamard's finite parts do not change, which follows from the equalities \eqref{estimate_t_by_s}\,--\,\eqref{estimate_numerator_W}, namely, 
\[
\begin{array}{rcl}
\displaystyle  H_N(\Gamma)&=&\displaystyle \lim_{\e\to0^+}\left(\const \iint_{(\Gamma\times \Gamma)\setminus\widetilde\Delta_\e}\,\frac{\,d\vect x_1\cdot d\vect x_2\,}{|\vect x_1-\vect x_2|}+\ciil\log\e\right),  \\[5mm]
\displaystyle H_W(\Gamma)&=&\displaystyle \lim_{\e\to0^+}\left(\const \iint_{(\Gamma\times \Gamma)\setminus\widetilde\Delta_\e}\,\frac{\,(\vect{\hat r}_{12}\cdot d\vect x_1)(\vect{\hat r}_{12}\cdot d\vect x_2)\,}{|\vect x_1-\vect x_2|}+\ciil\log\e\right). 
\end{array}
\]
We remark that such a phenomenon can be observed when the power of denominator in the integrand is smaller than $3$, as was explained for a similar functional in Remark 2.2.1 of \cite{O1}.

Next, to have faster convergence in numerical experments, it is better to increase the number of counter terms. The series expansion in $\e$, after dropping off the constant $\mu_0/(4\pi)$, is given by 

\begin{eqnarray}
&&\displaystyle  \iint_{(\Gamma\times \Gamma)\setminus\widetilde\Delta_\e}\,\frac{\,d\vect x_1\cdot d\vect x_2\,}{|\vect x_1-\vect x_2|}
=\displaystyle 2\mathcal{L}\log\frac1\e 
+\frac{11}{24}\left(\int_\Gamma \kappa(x)^2\,dx\right)\e^2 +O(\e^4), \label{Hadamard_N_2} \\[1mm]
&&\displaystyle \iint_{(\Gamma\times \Gamma)\setminus\widetilde\Delta_\e}\,\frac{\,(\vect{\hat r}_{12}\cdot d\vect x_1)(\vect{\hat r}_{12}\cdot d\vect x_2)\,}{|\vect x_1-\vect x_2|}
=\displaystyle 2\mathcal{L}\log\frac1\e
+\frac{5}{24}\left(\int_\Gamma \kappa(x)^2\,dx\right)\e^2 +O(\e^4). \nonumber 
\end{eqnarray}
Let us show that \eqref{Hadamard_N_2} follows from \eqref{estimate_t_by_s} and \eqref{estimate_numerator_N}. 
Fix a point $\vect{x}_1$. 
For a small positive number $b$ and $\e$ with $0<\e< b$, 
\[\begin{array}{l}
\displaystyle \int_{\e\le d_\Gamma(\vect x_1, \vect x_2)\le b}\frac{\vect t_{\vect x_1}\cdot \vect t_{\vect x_2}}{|\vect x_1-\vect x_2|}\,dx_2\\[4mm]
=\displaystyle \int_\e^b 
\left(1-\frac{\kappa^2}2s^2-\frac{\kappa\kappa'}2s^3+O(s^4)\right)\left[s\left(1-\frac{\kappa^2}{24}\,s^2-\frac{\kappa\kappa'}{24}s^3+O(s^4)\right)\right]^{-1}\,ds \\[4mm]
\displaystyle \phantom{=} +\int_{-b}^{-\e}
\left(1-\frac{\kappa^2}2s^2-\frac{\kappa\kappa'}2s^3+O(s^4)\right)\left[-s\left(1-\frac{\kappa^2}{24}\,s^2-\frac{\kappa\kappa'}{24}s^3+O(s^4)\right)\right]^{-1}\,ds \\[4mm]
\displaystyle =O(1)-2\log\e+\frac{11}{24}\kappa^2\cdot\e^2+O(\e^4).
\end{array}\]

\subsection{The self-inductance of solenoids}
Let us consider a solenoid $\Gamma_n=\Gamma_{r,\ell,n}$ that winds a cylinder of radius $r$ and length $\ell$ with a number of turns per unit length $n$ carrying a current $I$. 
Define the regularized self-inductance of $\Gamma$ in the sense of Neumann and of Weber by \eqref{Hadamard_N} and \eqref{Hadamard_W}. 

\begin{theorem}\label{thm_solenoids}
The asymptotic of the regularized self-inductances as $n$ goes to infinity is given by 
\begin{equation}\label{solenoid}
\begin{array}{l}
\displaystyle \lim_{n\to\infty}\frac{H_N(\Gamma_{r,\ell,n})}{n^2}
=\displaystyle \lim_{n\to\infty}\frac{H_W(\Gamma_{r,\ell,n})}{n^2} \\[4mm]
=\displaystyle \frac{8\mu_0}3\left[
-r^3+\frac18\left(
-\ell\left(\ell^2-4r^2\right)E\left(-\frac{4r^2}{\ell^2}\right)
+\ell\left(\ell^2+4r^2\right)K\left(-\frac{4r^2}{\ell^2}\right)
\right)
\right],
\end{array}\end{equation}
where $K(k)$ and $E(k)$ are the complete elliptic integrals of the first and second kind respectively:
\begin{equation}\label{elliptic}
E(k)=\int_0^{\frac\pi2}\sqrt{1-k\sin^2(t)}\,dt, \hspace{0.4cm}
K(k)=\int_0^{\frac\pi2}\frac1{\sqrt{1-k\sin^2(t)}}\,dt. \nonumber
\end{equation}

Let \eqref{solenoid} be denoted by $L_{r,\ell}$. 
The asymptotic of $L_{r,\ell}$ as $\ell$ goes to infinity is given by 
\begin{equation}\label{solenoid-l}
\lim_{\ell\to\infty}\left(
L_{r,\ell}-\mu_0\left(\pi r^2\ell-\frac83r^3
\right)
\right)=0.
\end{equation}
\end{theorem}

Thus the asymptotics of regularized self-inductance as $n$ and $\ell$ go to infinity is given by 
\[
H_N(\Gamma_{r,\ell,n})\sim H_W(\Gamma_{r,\ell,n}) \sim \mu_0 \pi r^2 n^2 \ell \hspace{0.4cm} (n,\ell \to \infty),
\]
as is expected. 

\begin{remark}\rm 
The right hand side of \eqref{solenoid} multiplied by $n^2$ has already appeared in \cite{BAB} and \cite{DO} (and in \cite{L} according to \cite{BAB}) as the limit of the mutual inductance of coaxial solenoids as they overlap. 
\end{remark}

\begin{proof}
(1) The first equality of \eqref{solenoid} follows from Theorem \ref{theorem1} (4) since $\mathcal{L}=\ell \sqrt{(2\pi n r)^2+1}=O(n)\,$.

(2) The second equality of \eqref{solenoid} can be proved as follows. 
Let $M=M_{r,\ell}$ be a cylinder with radius $r$ and length $l$ parametrized by cylindrical coordinates; 
\[
\vect p\colon[0,2\pi r]\times[-\ell/2,\ell/2]\ni(t,z)\mapsto(r\cos(t/r),\,r\sin(t/r),\,z)
.\]
Let $\vect v_{\vect x}$ be the unit tangent vector to the ``meridean circle'' of $M_{r,\ell}$ through $\vect x$ which is given by $\vect v_{\vect x}=\vect p_t=(\partial\vect p/\partial t)$. 
We show 
\begin{equation}\label{coil->2dim}
\lim_{n\to\infty}\frac{H_N(\Gamma_{n})}{n^2} 
=\const\iint_{M\times M} 
\frac{\vect v_{\vect x_1}\cdot \vect v_{\vect x_2}}{|\vect x_1-\vect x_2|}\,d\vect x_1 d\vect x_2
\end{equation}
by the following steps. 

Let $\e_0$ be any posotive number. 

(i) For any positive number $\delta$ there exists a natural number $n_0$ such that for nay point $\vect x$ in $M$ there holds
\[
\left|
\iint_{(\Gamma_{n}\times \Gamma_{n})\setminus N_\delta(\vect x)} 
\,\frac{\,d\vect x_1\cdot d\vect x_2\,}{|\vect x_1-\vect x_2|}
-\iint_{(M\times M)\setminus N_\delta(\vect x)} 
\frac{\vect v_{\vect x_1}\cdot \vect v_{\vect x_2}}{|\vect x_1-\vect x_2|}\,d\vect x_1 d\vect x_2
\right|<\e_0,
\]
where $N_\delta(\vect x)$ is a ``curved square neighbourhood'' of $\vect x=\vect p(t,z)$ in $M$ given by 
\[
N_\delta(\vect x)=\vect p\left((t-\delta,t+\delta)\times
\left((z-\delta,z+\delta)\cap[-\ell/2,\ell/2]\right)\right), 
\]
where we assume that the coordinate $\theta$ is considered modulo $2\pi$. 


(ii) There exists a positive number $\delta_2$ such that if $0<\delta\le\delta_2$ then for any $\vect x$ in $M$, 
\[\iint_{(M\times M)\,\cap N_\delta(\vect x)} 
\frac{\vect v_{\vect x_1}\cdot \vect v_{\vect x_2}}{|\vect x_1-\vect x_2|}\,d\vect x_1 d\vect x_2<\e_0.\]

(iii) There is a natural number $n_1$ such tha if $n\ge n_1$ then for any positive number $\delta$ with $0<\delta\le\delta_2$, 
\[
\frac1n\left|\,
\lim_{\e\to0^+}\left(\int_{[-\delta,-\e]\cup[\e,\delta]}
\frac{\vect p_t(t,0)\cdot\vect p_t(0,0)}{|\vect p(t,0)-\vect p(0,0)|}\,dt
+2\log\e\right)
+2\sum_{i=1}^{\lfloor \delta/n \rfloor}\int_{-\delta}^{\delta}
\frac{\vect p_t(t,i/n)\cdot\vect p_t(0,0)}{|\vect p(t,i/n)-\vect p(0,0)|}\,dt
\,\right|<2\e_0,
\]
where $\lfloor u \rfloor$ is the floor function that gives the greatest integer less than or equal to $u$. 

(iv) There is a natural number $n_2$ with $n_2\ge n_1$ such that for any positive number $\delta$ with $0<\delta\le\delta_2$ and for any point $\vect x$ in $M$, 
\[
\frac1n\left|\,
\lim_{\e\to0^+}\left(
\int_{(\Gamma_n\cap N_\delta(\vect x))\setminus\widetilde\Delta_\e}\frac{\vect t_{\vect x}\cdot\vect t_{\vect x_2}}{|\vect x-\vect x_2|}\,dx_2+2\log\e
\right)
\,\right|<3\e_0.
\]

Since the right hand side of \eqref{coil->2dim} is given by 
\[\begin{array}{l}
\displaystyle 
\const\int_0^{2\pi}\int_0^\ell\int_0^{2\pi}\int_0^\ell
\frac{\cos(\theta_1-\theta_2)}{\sqrt{4r^2\sin^2\frac{\theta_1-\theta_2}2+(z_1-z_2)^2}}\,r^2\,dz_1d\theta_1dz_2\theta_2 \\[4mm]
=\displaystyle 2\mu_0r^3\int_0^{\frac\pi2}\int_0^{\frac\ell r}\int_0^{\frac\ell r}
\frac{1-2\sin^2\theta}{\sqrt{4\sin^2\theta+(z_1-z_2)^2}}\,dz_1dz_2d\theta.
\end{array}\]
Some computation shows that it is equal to the right hand side of \eqref{solenoid}. 

\smallskip
(3) The equality \eqref{solenoid-l} follows from the following expansion; 
\[
\frac13\left(
-\ell\left(\ell^2-4\right)E\left(-\frac{4}{\ell^2}\right)
+\ell\left(\ell^2+4\right)K\left(-\frac{4}{\ell^2}\right)
\right)
=\pi\ell+\frac\pi2\cdot\frac1\ell+O\left(\frac1{\ell^3}\right).
\]
\end{proof}
\section{Appendix}
\subsection{The second residues of $\vect{F_N}$ and $\vect{F_W}$}

We can proceed the analysis of $F_N(z)$ and $F_W$ a bit more. 
\begin{proposition}\label{prop_nbd_arclength}
The second residues of $F_N$ and $F_W$ defined by \eqref{F_N(z)} and \eqref{F_W(z)} are given by 
\[
\Res(F_N,-3)=-\frac{3\mu_0}{16\pi}\int_\Gamma\kappa(x)^2 \,dx, \hspace{0.4cm}
\Res(F_W,-3)=-\frac{\mu_0}{16\pi}\int_\Gamma\kappa(x)^2 \,dx,
\]
where $\kappa$ denotes the curvature. 
\end{proposition}

\begin{proof}
The first equality follows from \eqref{series_varphi} and \eqref{residue_t^wphi}. 

The second equality can be proved similarly using 
\eqref{estimate_numerator_W}. 
%
\end{proof}

\subsection{When the power of the denominator is $2$}\label{subsection_OS1}
If we change the power of the denominator of \eqref{f_Neumann_link} from $1$ to $2$, we obtain ``{\sl average linking with random circles}'' (\cite{OS1}). 
Let $\mathcal{S}(1,3)$ be the set of oriented circles in $\RR^3$, $dC$ a measure on $\mathcal{S}(1,3)$ which is invariant under M\"obius transformations, and $\mbox{\sl lk}(C,\Gamma_i)$ the linking number of $\Gamma$ and an oriented circle $C$. Then we have 
\[
\int_{\Gamma_1}\int_{\Gamma_2}\,\frac{\,d\vect x_1\cdot d\vect x_2\,}{|\vect x_1-\vect x_2|^2}
=2\int_{\Gamma_1}\int_{\Gamma_2}\,\frac{\,(\vect{\hat r}_{12}\cdot d\vect x_1)(\vect{\hat r}_{12}\cdot d\vect x_2)\,}{|\vect x_1-\vect x_2|^2}
=\frac{3\pi}4\int_{\mathcal{S}(1,3)}\mbox{\sl lk}(C,\Gamma_1)\cdot \mbox{\sl lk}(C,\Gamma_2)\,dC.
\]
When $\Gamma_1$ and $\Gamma_2$ coincide, the above integrals blow up, and they can be regularized as 
\[
\lim_{\e\to0^+}\left(\iint_{(\Gamma\times \Gamma)\setminus\Delta_\e}\,\frac{\,d\vect x_1\cdot d\vect x_2\,}{|\vect x_1-\vect x_2|^2}-\frac{2\mathcal{L}}\e\right)
=\lim_{\e\to0^+}\left(\frac{3\pi}4\int_{\mathcal{S}_\e(1,3)}\mbox{\sl lk}(C,\Gamma)^2\,dC\ - \frac{3\pi\mathcal{L}}{2\e}\right),
\]
where $\mathcal{S}_\e(1,3)$ is the set of oriented circles with radius $r\ge \e$. 
The reader is referred to \cite{OS1} for the details.

Jun O'Hara

Department of Mathematics and Informatics,Faculty of Science, 
Chiba University

1-33 Yayoi-cho, Inage, Chiba, 263-8522, JAPAN.  

E-mail: ohara@math.s.chiba-u.ac.jp

\end{document}